\documentclass{amsart}

\usepackage{pdfsync}
\usepackage{mathpazo}
\usepackage{amscd}
\usepackage{amsmath}
\usepackage{amssymb}
\usepackage{amsthm}
\usepackage{epsf}
\usepackage{latexsym}
\usepackage{verbatim}
\usepackage{bbm}
\usepackage[utf8]{inputenc}
\usepackage{pdfsync}
\usepackage{enumerate}
\usepackage{amsmath,amsthm,mathrsfs}
\usepackage[english]{babel}
\usepackage{graphicx}
\usepackage{amssymb, xcolor}
\usepackage{epstopdf, enumerate}
\usepackage{amsfonts,amssymb,amsthm}
\usepackage{pifont, url, lscape}
\usepackage{amsfonts,amstext,amssymb,verbatim,epsfig}
\usepackage{pgf,tikz}
\usepackage{float}
\usetikzlibrary{arrows}
\usepackage{hyperref}
\usepackage{cite}
\usepackage{epstopdf}
\usepackage{color}
\usepackage{transparent}
\usepackage{subfigure}
\usepackage{xcolor}
\usepackage[normalem]{ulem}

\usepackage[left=1.1in,top=1.5in,right=1in,bottom=1in]{geometry}

\newcommand{\vol}{\text{vol}}

\newcommand{\E}{\mathbb{E}}

\newcommand{\R}{\mathbb{R}}

\newcommand{\C}{\mathbb{C}}

\newcommand{\Pa}{\mathbb{P}}

\newcommand{\n}{\mathbf{n}}
\newcommand{\tif}{\textnormal{if }}

\numberwithin{equation}{section}

\hypersetup{colorlinks=true,pdfborder=000,  citecolor  = blue, citebordercolor= {magenta}}
\hypersetup{linkcolor=purple}
\makeatletter
\let\reftagform@=\tagform@
\def\tagform@#1{\maketag@@@{(\ignorespaces\textcolor{purple}{#1}\unskip\@@italiccorr)}}
\renewcommand{\eqref}[1]{\textup{\reftagform@{\ref{#1}}}}
\makeatother
\newcommand{\Ima}{\textnormal{Im }}

\newcommand{\proj}{\textnormal{proj}}

\theoremstyle{plain}
\newtheorem{theorem}{Theorem}[section]
\newtheorem{corollary}[theorem]{Corollary}
\newtheorem{lemma}[theorem]{Lemma}

\newenvironment{Proof of lemma}{\noindent{\bf Proof of Lemma}}{\hfill$\Box$\newline}
\newenvironment{Proof of theorem}{\noindent{\bf Proof of Theorem}}{\hfill{\footnotesize${\square}$}\newline}
\newenvironment{Proof of theorems}{\noindent{\bf Proof of Theorems}}{\hfill$\Box$\newline}
\newenvironment{Proof of proposition}{\noindent{\bf Proof of Proposition}}{\hfill$\Box$\newline}
\newenvironment{Proof of propositions}{\noindent{\bf Proof of Propositions}}{\hfill$\Box$\newline}
\newenvironment{Proof of exercise}{\noindent{\it Proof of Exercise:}}{\hfill$\Box$}
\newenvironment{Acknowledgements}{\noindent{\bf Acknowledgement}}

\begin{document}
	
	\title{On the number of equilibria with a given number of unstable directions}\author{Xavier Garcia}
	\address{Department of Mathematics\\
		Northwestern University\\
		Evanston, IL 60208 USA}
	\email{sphinx@math.northwestern.edu}
	\maketitle

	\begin{abstract}
We compute the large-dimensional asymptotics for the average number of equilibria with a fixed number of unstable directions for random Gaussian ODEs on a sphere. We also discuss the effects that the value of the Lagrange multiplier of the vector field has on the number of such equilibria. This was a problem posed by Fyodorov~\cite{Fyo2}.
\end{abstract}

\section{Introduction and main results} 

Over the years, there has been an interest in understanding the dynamics of random Gaussian ordinary differential equations (ODEs) on high dimensional spheres, starting with the work of Cugliando et al. \cite{CKL} and most recently in the work of Fyodorov \cite{Fyo1},\cite{Fyo2}. The standard setup is to consider a first order ODE \begin{equation*}
\frac{dx}{dt} = F(x)
\end{equation*}
where $F$ is a random vector field on $S^{N-1}(\sqrt{N})$ and attempt to describe the behavior of the possible solutions as $N \rightarrow \infty$. One natural starting point is to count the number $\mathcal{N}_{tot}$ of equilibrium points and to study the large-dimensional asymptotics of this quantity. This is the content of Fyodorov's work in ~\cite{Fyo2}. In this paper, we classify the equilibrium points by stability. For an equilibrium point $\sigma$ and a neighborhood $U$ around $\sigma$, we choose coordinates on $U$ and the tangent space $TU$ such that $\sigma = 0 \in \R^{N-1}$ and write $F$ locally as a function $F : \R^{N-1} \rightarrow \R^{N-1}$ denoted by $$F(x) = (c_1(x), ... \,, c_{N-1}(x)).$$ We say $\sigma$ is an equilibrium point with $m$ unstable directions if the Jacobian matrix $JF(\sigma) := (\partial_i c_j(0))$ has exactly $m$ eigenvalues with non-negative real part. While the explicit formula for $JF(\sigma)$ depends on the coordinates chosen, its eigenvalues do not. Our focus on this paper will be on the number $\mathcal{N}_{m}$ of equilibria with $m$ unstable directions and its related large-dimensional asymptotics. In the case of a gradient flow (known in the literature as relaxational dynamics), this problem has been studied in great detail, as can be found in Auffinger, Ben Arous and \v{C}ern\'y ~\cite{ABC} and Fyodorov~\cite{Fyo1}. It is the purpose of this paper to obtain the asymptotics in the non-relaxational case. We compute the exponential rate of $\E \mathcal{N}_{m}$ under very general conditions as the dimension goes to infinity. The methods in this paper will follow the tried-and-true approach of relating the problem to a matrix integral through the Kac-Rice formula, then invoking a large deviation principle (LDP) to obtain the asymptotics. To be more precise, let us make our assumptions on $F$ explicit.

Classically, $F$ is viewed as a map $F : S^{N-1}(\sqrt{N}) \subset \R^{N} \rightarrow \R^{N}$ by identifying $S^{N-1}(\sqrt{N})$ with the usual $(N-1)$ dimensional sphere in $\R^N$ centered at 0 with radius $\sqrt{N}$ and for $x \in S^{N-1}(\sqrt{N}) \subset \R^N$, $$T_xS^{N-1}(\sqrt{N}) = \{ v \in \R^N : \langle v, x \rangle = 0 \}$$where $\langle \cdot, \cdot \rangle$ denotes the usual inner product in $\R^N$. With this framework, the vector fields considered in Fyodorov~\cite{Fyo2} take the form: \begin{equation*}
F(x) = -\lambda(x)x + f(x) + h
\end{equation*}where $h = (h_1,...,h_N)$ is an $N$-dimensional Gaussian vector with covariance structure $$\E[h_i h_j] = \sigma^2 \delta_{ij}$$for some $\sigma > 0$, $\delta_{ij}$  the usual Kronecker delta and $f$ is an $N$-dimensional smooth Gaussian field with covariance kernel\begin{equation*}\label{eq:Phi}\E[f_i(x) f_j(y)] = \delta_{ij} \Phi_1 \left(\frac{ \langle x, y \rangle}{N}\right) + \frac{x_jy_i}{N} \Phi_2 \left(\frac{ \langle x, y \rangle}{N}\right)\end{equation*}where $\Phi_1$ and $\Phi_2$ are smooth functions satisfying \begin{equation}\label{eq:phi}0 < \Phi_1(1) < \Phi'_1(1), -\Phi_1(1) \leq \Phi_2(1) \leq \Phi_1(1).\end{equation}The Lagrange multiplier $\lambda$ is chosen so that the vector belongs to $T_xS^{N-1}(\sqrt{N})$. Explicitly,$$\lambda(x) = \frac{1}{N} \langle x, f(x) + h \rangle.$$ 
We also make the added assumption that $h$ is independent of $f$. Before stating our results, we define two quantities which will play an important role in our analysis:$$\tau = \frac{\Phi_2(1)}{\Phi'_1(1)}, \text{   } b^2 = \frac{\sigma^2 + \Phi_1(1)}{\Phi'(1)}.$$The restrictions given by (\ref{eq:phi}) imply so $-1 < \tau \leq 1$ and $b^2 + \tau \geq 0$. We additionally require that $b^2 + \tau > 0$ and restrict ourselves to the non-gradient case $\tau \neq 1$. 

In our first main result, we compute the exponential rate of the expected number $\E \mathcal{N}_m$ of critical points with $m$ unstable directions:

\begin{theorem}\label{mainthm}For $ b < 1$ and $-1 < \tau < 1$, we have:\begin{equation*}
	\lim_{N \rightarrow \infty} \frac{1}{N} \log \E \mathcal{N}_{m} = \log \frac{1}{b} - \frac{(1-b^2)(1+\tau)}{2(b^2 + \tau)}.
	\end{equation*}
\end{theorem}

This result says that there exists a curve $\tau(b)$ given explicitly by $$\tau(b) = - \frac{2b^2 \log b + 1 - b^2}{2 \log b + 1 - b^2}$$such that if $\tau > \tau(b)$, we have exponentially abundant equilibria with $m$ unstable directions and if $\tau < \tau(b)$ the probability of finding such equilibria is exponentially small. We remark that the case $b > 1$ is the ``topologically trivial" case with only two equilibrium points in the limit (see Fyodorov~\cite{Fyo2}), so we will omit the analysis of this case. 

 We are also interested in the case when we have a diverging number of unstable directions. Let $U_{\tau}$ denote the uniform distribution on the ellipse $$E_{\tau} =\left\{(x,y) : \frac{x^2}{(1+\tau)^2} + \frac{y^2}{(1-\tau)^2} \leq 1 \right\},$$ $\gamma \in (0,1)$, and define $s_{\gamma} \in (-1-\tau,1+\tau)$ to be the unique number such that $$U_{\tau} ( \textnormal{Re } z \geq s_{\gamma}) = \gamma.$$
 \begin{theorem}\label{divind}
Let $m(N)$ be a sequence of integers which satisfy $\frac{m(N)}{N} \rightarrow \gamma \in (0,1)$. Then, $$\lim_{N \rightarrow \infty} \frac{1}{N} \log \E \mathcal{N}_{m(N)} = \log \frac{1}{b} - \frac{1-b^2}{2(b^2+\tau)(1+\tau)} s^2_{\gamma}.$$
 \end{theorem}

Note that this quantity is maximized at $s_{\gamma} = 0$ which occurs at $\gamma = 1/2$. This allows us to recover estimates on the average total number of equilibria, since {\sc Theorem} \ref{divind} implies $$\E[\mathcal{N}_{m(N)}] \leq \E[\mathcal{N}_{tot}] \leq (N+1) \E[\mathcal{N}_{m(N)}]$$ for $N$ large and $m(N) / N \rightarrow 1/2$. Thus, both $\E[\mathcal{N}_{tot}]$ and $\E[\mathcal{N}_{m(N)}]$ have the exponential rate given by $\log \frac{1}{b}$ agreeing with {\sc Proposition} 2.5 of Fyodorov~\cite{Fyo1}.

Our next result details the relationship between the value of the Lagrange multiplier and critical points. For a Borel set $B \subset \R$, define $\mathcal{N}_{m}(B)$ to be the number of equilibria with $m$ unstable directions whose Lagrange multiplier has values in $B$.

\begin{theorem}\label{lagmult}
For $- \infty \leq c < d \leq \infty$ we have
\begin{equation*}
\lim_{N \rightarrow \infty} \frac{1}{N} \log \E \mathcal{N}_m(c,d) = \begin{cases} \log \frac{1}{b} - \frac{(1-b^2)(1+\tau)}{2(b^2 + \tau)}  & \tif c < (1+\tau) \sqrt{\Phi'_1(1)} < d \\
\log \frac{1}{b} -  \frac{(1-b^2)c^2}{2\Phi'_1(1)(b^2 + \tau)(1+\tau)} - (m+1)I_{\tau} \left( \frac{1}{\sqrt{\Phi'_1(1)}}c \right) & \tif c > (1+\tau) \sqrt{\Phi'_1(1)}\\
-\infty & \tif d < (1 + \tau) \sqrt{\Phi'_1(1)}  \\
\end{cases}
\end{equation*}
\end{theorem}
where $I_{\tau}$ is defined in {\sc Lemma} \ref{logpot} below. Hence, the probability of finding equilibria with $m$ unstable directions and a Lagrange multiplier less than $(1+\tau)\sqrt{\Phi'_1(1)}$ becomes exponentially small. Since the function $$c \mapsto  \frac{(1-b^2)c^2}{2\Phi'_1(1)(b^2 + \tau)(1+\tau)} + mI_{\tau} \left( \frac{1}{\sqrt{\Phi'_1(1)}}c \right)$$is increasing, unbounded and attains zero at $(1+\tau) \sqrt{\Phi'_1(1)}$, there is a unique point $z_0 > (1+\tau)\sqrt{\Phi'_1(1)}$ for which it is equal to $\log \left( \frac{1}{b} \right)$. The probability of finding equilibria with $m$ unstable directions and a Lagrange multiplier larger than $ z_0$ also becomes exponentially small. By the symmetry of the problem (and more explicitly by {\sc Theorem} \ref{eq:aux}), we have that $\mathcal{N}_{m}(B) = \mathcal{N}_{N-m}(-B)$ so we can make analogous statements about equilibria with $m$ stable directions. 

This paper is organized as follows. {\sc Section} 2 contains some preliminary results on the logarithm potential of the uniform distribution $U_{\tau}$ on the ellipse $E_{\tau}$ that will be used in the next section. In { \sc Section} 3, we introduce the Gaussian Elliptic Ensemble and discuss some of its properties. In {\sc Section} 4 we prove a large deviation result for the eigenvalue of the Gaussian Elliptic Ensemble with $m$th largest real part.  The goal of {\sc Section} 5 is to relate $\E \mathcal{N}_{m}(B)$ to a matrix integral involving the Gaussian Elliptic Ensemble. The proof of the main assertion in {\sc Section} 5 is postponed to {\sc Section} 6. Finally, in  {\sc Section} 7 we provide the proofs for the main results stated in this section.

\begin{Acknowledgements}
\text{  }	I would like to thank Antonio Auffinger for introducing me to the problems solved in this paper.
\end{Acknowledgements}

\section{The logarithmic potential on an ellipse}

In this section, we recall some properties of the logarithmic potential for the uniform distribution $U_{\tau}$ on the ellipse $E_{\tau}$ $$E_{\tau} =\left\{(x,y) : \frac{x^2}{(1+\tau)^2} + \frac{y^2}{(1-\tau)^2} \leq 1 \right\}, -1 < \tau \leq 1$$The logarithmic potential of $U_{\tau}$ is defined as a function $\phi_{\tau} : \R^2 \rightarrow \R$ explicitly given by $$\phi_{\tau}(x,y) = \int_{E_{\tau}} \log |x + iy - w | \textnormal{ } U_{\tau}(dw).$$  We summarize the properties of $\phi_{\tau}$  we will need in the lemma below. 

\begin{lemma}\label{logpot}
	On the set $\{x \geq 1+ \tau,  y \geq 0\}$:
	
	\begin{enumerate}
		\item $\phi_{\tau}(x,0) + \frac{1}{2}- \frac{x^2}{2(1+\tau)}  = -I_{\tau}(x)$ where $$I_{\tau}(x) := \begin{cases} \frac{1}{2(1+\tau)}x^2 - \frac{x(x-\sqrt{x^2 - 4 \tau})}{4 \tau}  -  \log(\frac{x + \sqrt{x^2-4\tau}}{2}) & \tif \tau \neq 0 \\ -\log x + \frac{1}{2}x^2 - \frac{1}{2} &\tif  \tau = 0
		\end{cases}$$
		\item  $\partial_x \phi_{\tau}(x,y) \leq \frac{x}{1+\tau}$
		\item $\partial_y \phi_{\tau }(x,y) \leq \frac{y}{1-\tau}$
	\end{enumerate}
\end{lemma}

\begin{proof}
	
	In the case $\tau = 0$, we can compute $\phi_0 (x,y) = \frac{1}{2} \log(x^2+y^2)$ from which we can verify all the results instantly. Henceforth, we shall assume $\tau \neq 0$. 
	
	From page 9 of Bell et. al \cite{BEF}, we have for $z = x  + iy$:\begin{equation*}\label{Cauchy}
	\partial_x \phi_{\tau}(x,y) - i \partial_y \phi_{\tau}(x,y)  = \int_{E_{\tau}} \frac{1}{z-w} U_{\tau}(dw) = \frac{1}{2\tau} \left( z - \sqrt{z^2 - 4 \tau}\right).
	\end{equation*}
	In particular, this implies $$\partial_x \phi_{\tau}(x,0) = \frac{1}{2 \tau} \left(x - \sqrt{x^2-4\tau} \right).$$
	By integration, we obtain $$\phi_{\tau}(x,0) - \phi_{\tau}(1+\tau,0) =  \frac{1}{2(1+\tau)}x^2 -\frac{1+\tau}{2} -I_{\tau}(x).$$ By {\sc Lemma }5.3.12 of Hiai and Petz~\cite{HP}, we know that $\phi_{\tau}(1+\tau,0) = \frac{\tau}{2}$ thus yielding the first statement. Statements (2) and (3) follow from the proof of {\sc Lemma} 5.3.12 of Hiai and Petz\cite{HP}.
\end{proof}
Since $y \mapsto \phi_{\tau}(x,y)$ is an even function of $y$ for all $x$, we have the following corollary to {\sc Lemma} \ref{logpot}:

\begin{corollary}\label{logpotineq}Define $\Psi_{\tau}(x,y) = \phi_{\tau}(x,y) - \frac{x^2}{2(1+\tau)} - \frac{y^2}{2(1-\tau)}$. For $x \geq 1+\tau$, $y \in \R$ and $t \in (-1,1)$, $$\sup_{u \geq x,v \geq y}  \Psi_{\tau}(u,v) = \Psi_{\tau}(x,y) \leq \Psi_{\tau}(x,0).$$
	
\end{corollary}
\section{On the Gaussian Elliptic Ensemble}

We define the Gaussian Elliptic Ensemble(GEE) as an $N \times N$ random matrix $X$ whose entries are mean zero Gaussian random variables with covariance structure $$\E[X_{ij} X_{lk}] := \E[X_{ij} X_{lk}] = \frac{1}{N}( \delta_{il} \delta_{jk} + \tau \delta_{ik} \delta_{jl}), -1 < \tau \leq 1$$We can write the density of this measure against the Lebesgue measure $dX$ on the space of real $N \times N$ matrices:$$ \Pa_N(dX) = \frac{1}{Z_N(\tau)} \exp \left( -\frac{N}{2(1 - \tau^2)} \textnormal{Tr}(XX^T - \tau X^2) \right) dX$$where $$Z_N(\tau) = 2^{N/2} \pi^{N(N+1)/2} (1+\tau)^{N(N+1)/4}(1-\tau)^{N(N-1)/4}N^{\frac{N^2}{2}}.$$We shall also think of the eigenvalues of $X$ as ordered by decreasing real parts i.e., we will denote them by $\lambda_1(X),...\,,\lambda_N(X)$ with $\lambda_1(X)$ having the largest real part, with the understanding that if we have complex eigenvalues, we list the ones with positive imaginary part first. Almost surely, this is a well-defined ordering. When there is no room for confusion, we will usually drop the $X$ from the $\lambda(X)$ to ease the notation.

The purpose of this section is to state some properties of the GEE as well as to  prove a large deviation principle for the eigenvalue with the $m$th largest real part. We first require a formula for the joint distribution of the eigenvalues. Since a matrix distributed like the GEE will have real eigenvalues with positive probability, the joint distribution of the eigenvalues will not be absolutely continuous with respect to the Lebesgue measure on $\C^N$. Nevertheless, we can write out formulas if we restrict ourselves to the sets $$S_k = \{ X \textnormal{ has exactly } k \textnormal{ real eigenvalues}\}.$$ If $A \in S_k$, then we can write the eigenvalues of $A$ as $\sigma_1,...,\sigma_k, x_1 \pm i y_1,...,x_{\frac{N-k}{2}} \pm i y_{\frac{N-k}{2}}$. This suggests defining the measure $$\mu^{(N,k)}(d \sigma, dx, dy) = 2^{(N-k)/2} \prod_{i=1}^k d \sigma_i \prod_{j=1}^{\frac{N-k}{2}} dx_j dy_j$$ on the set $$\{ (\sigma,x,y) \in \R^k \times \R^{N-k} : \sigma_1 > ... > \sigma_k, x_1 > ... > x_{\frac{N-k}{2}}, y_i \geq 0 \, \forall i\}.$$ If we define the measure $\Pa_{N,k}$ by $\Pa_{N,k}( V) = \Pa_{N}(\lambda(X) \in V, X \in S_k)$ for any Borel set $V \subset \C^{N}$, then $\Pa_{N,k}$ is absolutely continuous with respect to $\mu^{(N,k)}$ and the density is given by
	\begin{equation*}\label{SK}
	\Pa_{N,k}(d \sigma, dx, dy) =  \frac{1}{K_N(\tau)} |\Delta(\sigma,x \pm iy)| \prod_{i=1}^{k} e^{-\frac{N}{2(1+\tau)}\lambda_i^2} \prod_{j=1}^{\frac{N-k}{2}} e^{-\frac{N}{1+\tau}(x_j^2 - y_j^2)} \textnormal{erfc} \left( \sqrt{\frac{2N}{1-\tau^2}} y_j \right) \mu^{(N,k)}(d \sigma, dx, dy)
	\end{equation*}where \begin{equation}\label{K}K_N(\tau) = 2^{N(N+1)/4} (1+\tau)^{N/2} N^{\binom{N+1}{2}/2} \prod_{j=1}^{N} \Gamma(j/2),\end{equation} $\Delta(\sigma,x \pm iy)$ is the Vandermonde polynomial $$|\Delta(\sigma, x \pm iy)|^2 = \prod_{u,v \in S, u \neq v} | u - v|, S = \{\sigma_1,...,\sigma_k, x_1 \pm i y_1,...,x_{\frac{N-k}{2}} \pm i y_{\frac{N-k}{2}}\}$$and erfc is the complementary error function $$\textnormal{erfc}(x) = \frac{2}{\sqrt{\pi}} \int_x^{\infty} e^{-t^2} dt.$$See Lehmann and Sommers~\cite{LS}.

We will find it convenient to rewrite the integrand in a more compact way, namely as $$\frac{1}{K_N(\tau)} | \Delta(\lambda)| \exp \left(-\frac{N}{2(1+\tau)}\sum_{j=1}^N \lambda^2_j \right) \prod_{j=1}^N \sqrt{\textnormal{erfc} \left( \sqrt{\frac{2N}{1-\tau^2}} |\Ima \lambda_j| \right)} $$ with the understanding that $\lambda_j = x_j + iy_j$ if $\lambda$ is complex, and $\lambda_j = \sigma_j$ if it's real. The disadvantage of this form is that it obscures the dependence of $k$ and the symmetry obtained from the fact that the complex eigenvalues come in pairs. 

We find it useful to get rid of the ordering of the eigenvalues by their real parts. We can do this by simply replacing our region of integration by $\{(\sigma,x,y) : y_i \geq 0 \, \forall i\}$ and dividing by the factor of $k! \left( \frac{N-k}{2} \right)!$

\section{Large deviation principle for $\lambda_m(X)$}

In this section, we establish a large deviation type of result for  $\lambda_m(X)$. 

\begin{theorem}\label{largedev}
	Under $\Pa_N$, the quantity $\lambda_m \cdot \mathbbm{1}_{\R}(\lambda_m)$ satisfies a large deviation principle with speed $N$ and good rate function $mI_{\tau}$ where $$I_{\tau}(x) = \begin{cases} \infty &\tif x < 1 + \tau \\  \frac{1}{2(1+\tau)}x^2 - \frac{x(x-\sqrt{x^2 - 4 \tau})}{4 \tau}  -  \log(\frac{x + \sqrt{x^2-4\tau}}{2}) & \tif x \geq 1 + \tau, \tau \neq 0 \\ -\log x + \frac{1}{2}x^2 - \frac{1}{2} &\tif x \geq 1 + \tau, \tau = 0.
	\end{cases}$$
	
\end{theorem}

In order to prove Theorem \ref{largedev}, we will first need an exponential tightness result:

\begin{lemma}[Exponential tightness from the right]\label{tightness} The following limit holds:\begin{equation*}
	\lim_{M \rightarrow \infty} \lim_{N \rightarrow \infty} \frac{1}{N} \log \Pa_N \left( |\textnormal{Re } \lambda_1(X) | \geq M \textnormal{ or } \max_j \Ima \lambda_j(X)  > M \right) = - \infty.
	\end{equation*}
\end{lemma}
\begin{proof}We shall only prove \begin{equation*}
	\lim_{M \rightarrow \infty}\lim_{N \rightarrow \infty} \frac{1}{N} \log \Pa_N \left( \max_i \Ima \lambda_i > M \right) = - \infty
	\end{equation*} since an analogous argument will provide a proof for the statement involving $|\textnormal{Re } \lambda_1|$.  We shall make use of the following four inequalities: \begin{enumerate} \item There exists $\alpha > 0$ such that for $|z|$ for sufficient large and $w \in \C$,  \begin{equation*}\label{exp}|z - w||\bar{z} - w| \exp \left[ -\frac{|w|^2}{1+|\tau|} \right] \leq (|z| + | w|)^2 \exp \left[ -\frac{|w|^2}{1+|\tau|}\right] \leq \alpha^2 \exp \left[\frac{|z|^2}{2(1+|\tau|)} \right]. \end{equation*} 
		
		\item For any $z = x + iy \in \C$, \begin{equation}\label{erfc} \exp \left[ -\frac{N}{2(1+\tau)} (z^2 + \bar{z}^2) \right] \textnormal{erfc} \left[ \sqrt{\frac{2N}{1-\tau^2}} |y| \right] \leq \frac{\exp \left[ \frac{-N|z|^2}{1-\tau} \right]}{ \sqrt{ \frac{2N}{1-\tau^2}y^2} + \sqrt{ \frac{2N}{1-\tau^2}y^2 + \frac{4}{\pi}} } \leq \exp \left[ \frac{-N|z|^2}{1+|\tau|} \right].\end{equation}
		
		\item For any $y \in \R$, \begin{equation*}\label{erfc2}
		\textnormal{erfc} \left[ \sqrt{ \frac{2N}{1-\tau^2}} |y| \right] \leq \sqrt{\pi}\textnormal{erfc} \left[ \sqrt{ \frac{2(N-2)}{1-\tau^2}} | y| \right] \textnormal{erfc} \left[ \sqrt{ \frac{4}{1-\tau^2}} |y| \right] \left( \sqrt{ \frac{4y^2}{1-\tau^2}} + \sqrt{ \frac{4y^2}{1-\tau^2} + 2} \right)
		\end{equation*}
		
		\item There exists a $C$ such that for all $y \in \R$, \begin{equation*}
		\sqrt{\pi}\frac{\sqrt{ \frac{4y^2}{1-\tau^2}} + \sqrt{ \frac{4y^2}{1-\tau^2} + 2}}{\sqrt{ \frac{4}{1-\tau^2}y^2} + \sqrt{ \frac{4}{1-\tau^2}y^2 + \frac{4}{\pi}}} \leq C^{-2}
		\end{equation*}
		
	\end{enumerate}
	We shall establish an inequality for $\Pa_{N,k}( |\Ima \lambda_l| > M)$ for an arbitrary $l$ since $$\Pa_{N} \left( \max_j \Ima \lambda_j > M \right) \leq \sum_k \sum_l \Pa_{N,k}( |\Ima \lambda_l| > M).$$ We proceed as follows:
	\begin{align*}
	\Pa_{N,k}( |\Ima \lambda_l| > M) &= \frac{1}{K_N(\tau)} \int \exp \left( - \frac{N}{2(1+\tau)} \sum_{j=1}^m \lambda_j^2 \right) \prod_{j=1}^N \sqrt{\textnormal{erfc} \left( \sqrt{\frac{2N}{1+\tau}} |\Ima \lambda_j| \right)} |\Delta(\lambda)| \, d \mu^{(N,k)} \\
	&\leq \frac{\alpha^{N+k-2}}{C^{N-k-2}} \frac{K_{N-2}(\tau)}{K_N(\tau)} \int_{-\infty}^{\infty} \int_{M}^{\infty} 2y_l\exp \left( -\frac{(N-k+2)(x^2_l + y^2_l)}{2(1+|\tau|)} \right) \Pa_{N-2,k} (S_k) \, dx_l dy_l
	\end{align*}
	where the first integral is over $\{ \Ima \lambda_l > M\}$ and the quantities $S_k$, $\Pa_{N,k}$ and $\mu^{(N,k)}$ are as in {\sc Lemma} \ref{SK}. The inequality follows from the fact that we can split the exponential by $$\exp \left( -\frac{N}{2(1+\tau)}\lambda^2_j \right) = \exp \left( -\frac{N-2}{2(1+\tau)}\lambda^2_j \right) \exp \left( -\frac{2}{2(1+\tau)}\lambda^2_j \right) $$ and use inequality (3) to split the erfc term into three parts, whose first factor coupled with the first factor of the exponential and the appropiate factors from the Vandermonde polynomial yields the density for $\Pa_{N-2,k}$. We then use inequality (2) with $z = \sqrt{ \frac{2}{N}} \lambda_j$ for $j=1,...,N$  to get rid of our remaining erfc factors. This also allows us to bound the term coming from the last factor on the right hand side of inequality (3) by $C^2$. We then use inequality (1) to bound the remaining terms from the Vandermonde polynomial with the exception of $|\lambda_l - \bar{\lambda}_l|$, which is equal to $2 y_l$. Finally, we invoke inequality (2) again to deal with the remaining erfc factors involving $\lambda_l$. After some simplification, the inequality follows. Without loss of generality, we can assume $C \leq 1$ and hence $C^{N-k-2} \geq C^{N-2}$ and $\alpha^{N+k-2} \geq \alpha^{2N-2}$. We can sum over $l$ then over $k$ to obtain: \begin{align*}\Pa_N \left( \max_j \Ima \lambda_j(X) > M \right) &\leq C^{2-N} \alpha^{2N-2}\frac{K_{N-2}(\tau)}{K_N(\tau)} N^2\int_{-\infty}^{\infty} \int_M^{\infty} 2y e^{-\frac{(N+2)}{2(1+|\tau|)}(x^2 + y^2)} dy dx \\
	&\leq \frac{2(1+|\tau|)}{N-k+2} C^{2-N}\alpha^{2N-2}\frac{K_{N-2}(\tau)}{K_N(\tau)} N^2 e^{- \frac{N-k+2}{2(1+|\tau|)}M^2}\int_{-\infty}^{\infty}  e^{-\frac{(N-k+2)}{2(1+|\tau|)}x^2} dx \\
	&= \sqrt{\frac{2\pi(1+|\tau|)}{N-k+2}} C^{2-N}\alpha^{2N-2} \frac{K_{N-2}(\tau)}{K_N(\tau)} N^2 e^{-\frac{N-k+2}{2(1+|\tau|)}M^2}.
	\end{align*}
	where $N^2$ factors comes from the fact that the number of terms is bounded by $N^2$. The final exponential term will yield the desired asymptotics.
\end{proof}
Finally, we'll need the two more preliminary results.

(a) Define $\mathcal{P}(\C)$ to be the space of probability measures on $\C$ which are invariant under complex conjugation endowed with a metric compatible with the usual weak convergence of measures. The empirical measure of the eigenvalues $L_N = \frac{1}{N} \sum_{j=1}^N \delta_{\lambda_j(X)}$ with respect to $\Pa_N$ satisfies an LDP of speed $N^2$ on $\mathcal{P}(\C)$ whose rate function is minimized at the uniform distribution $U_{\tau}$ on the ellipse $E_{\tau}$. The proof of this LDP for the case $\tau = 0$ can be found in Ben Arous and Zeitouni~\cite{BGZ} but the same argument extends to $\tau \in (-1,1).$

(b) The functional
\begin{equation*}
\Psi(\mu,x+iy) = \int_{\C} \log |x + iy - z| d \mu(z) - \frac{x^2}{2(1+\tau)} - \frac{y^2}{2(1-\tau)} 
\end{equation*}
defined on $\mathcal{P}(\C) \times \C$ is upper-semicontinuous when restricted to $\mathcal{P}(B_M) \times B_M$ for any $M > 0$, and in fact is continuous when restricted to $\mathscr{P}(B_M) \times \left( B_M \cap \{ z : \textnormal{Re } z > x \} \right)$ for any $x > 1 + \tau$.  The distribution $U_{\tau}$ is related to the rate function $I_{\tau}$ by 
\begin{equation}\label{ratefunc}
I_{\tau}(x) = -\Psi(U_{\tau},x)-\frac{1}{2}
\end{equation}as per (1) of {\sc Lemma} \ref{logpot} since $\Psi(U_{\tau},x) = \Psi_{\tau}(x,0)$. 

With these preparations in order, we can begin the proof of {\sc Theorem} \ref{largedev}.

\begin{proof}[Proof of Theorem \ref{largedev}]
	It is obvious that $I_{\tau}$ is a good rate function.  Our theorem will follow if we can prove the following equalities:
	
	\begin{enumerate}
		\item $\lim_{N \rightarrow \infty} \frac{1}{N} \log \Pa_N( \lambda_{m} \in [0,x)) = -\infty$ for $0< x < 1+\tau$.
		\item $\lim_{N \rightarrow \infty} \frac{1}{N} \log \Pa_N( \lambda_{m} \in [x, \infty)) = -m I_{\tau}(x)$ for $x > 1+\tau$.
	\end{enumerate}
	
	To prove the  first equality, we note that by definition, $\textnormal{Re }\lambda_m(X) < x$ for some $x < 1+\tau$ is equivalent to $L_N(z : x \leq \textnormal{Re }z < 1+\tau) \leq \frac{m-1}{N}$. Since $\mu_{\tau}[x,1+\tau) > 0$, there exist constants $K,\kappa > 0$ such that if $x < 1 + \tau$ then 
	$$\Pa_N( \textnormal{Re } \lambda_m(X) < x) \leq \Pa_N \left(L_N( z: x \leq \text{ Re } z < 1 + \tau)  \leq \frac{m-1}{N} \right) \leq Ke^{-\kappa N^2}.$$
	 This inequality implies the result.
	
	We now turn to the proof of the second equality. In view of {\sc Lemma} \ref{tightness}, the second equality is equivalent to the following equality for sufficiently large $M$ satisfying $1 + \tau < x < M$:\begin{align*}\lim_{N \rightarrow \infty} \frac{1}{N} \log \Pa_N \left( \lambda_m(X) \in [x,M], |\textnormal{Re } \lambda_1(X)| \leq M, \max_j\Ima \lambda_j(X) \leq M \right) = -mI_{\tau}(x).
	\end{align*}
	
	We will estimate this by decomposing $\Pa_{N}$ into a sum of the $\Pa_{N,k}$ for admissible $k$. To that end, fix $k$ for the time being, introduce new variables $\tilde{\lambda}_j = \sqrt{\frac{N}{N-m}} \lambda_j \textnormal{ for } 1 \leq j \leq N$. On the larger set of $$\left \{ \tilde{\lambda}_m \in [x,2M], \max_j \left( \textnormal{Im }\tilde{\lambda}_j(X), |\textnormal{Re } \tilde{\lambda}_1|\right) \leq 2M \right \} \supset \left \{ \lambda_m(X) \in [x,M], \max_j \left(\Ima \lambda_j(X), |\textnormal{Re } \lambda_1(X)| \right) \leq M \right \},$$ we have the $|\tilde{\lambda}_j - \tilde{\lambda}_i| \leq 4 \sqrt{2}M$ and hence
	
	\begin{align*}
	\Pa_{N,k}(d \lambda) &= \frac{1}{K_N(\tau)} | \Delta(\lambda)| \exp \left(-\frac{N}{2(1+\tau)}\sum_{j=1}^N \lambda_j^2 \right) \prod_{j=1}^m \sqrt{\textnormal{erfc} \left( \sqrt{\frac{2N}{1-\tau^2}} |\Ima \lambda_j| \right)} d \mu^{(N,k)} \\
	&=   \frac{K_{N-m}(\tau)}{K_N(\tau)}\sum_{l} \prod_{1 \leq i < j \leq m} | \tilde{\lambda}_j - \tilde{\lambda}_i|  \prod_{j=1}^m e^{-\frac{N-m}{2(1+\tau)} \tilde{\lambda}^2_j} \sqrt{\textnormal{erfc}\left(\sqrt{\frac{2(N-m)}{1-\tau^2}} |\Ima \tilde{ \lambda}_j|\right)} \, \mu^{(m,l)}(d \tilde{\lambda}_1,...,d\tilde{\lambda}_m) \times \\
	& \qquad \qquad \qquad \frac{(k-l)! \left( \frac{N-m+l-k}{2} \right)!}{k!\left(\frac{N-k}{2}\right)!}\left(\frac{N-m}{N}\right)^{\frac{N(N+1)}{4}}  \prod_{i=1}^m \prod_{j=m+1}^{N} |\tilde{\lambda}_i - \tilde{\lambda}_j|  \Pa_{N-m,k-l}(d\tilde{\lambda}_{m+1},...,d\tilde{\lambda}_N) \\ 
	&\leq (4\sqrt{2}M)^{\frac{m(m-1)}{2}}  \frac{K_{N-m}(\tau)}{K_N(\tau)} \left(\frac{N-m}{N}\right)^{\frac{N(N+1)}{4}}\sum_l \frac{(k-l)! \left( \frac{N-m+l-k}{2} \right)!}{k!\left(\frac{N-k}{2}\right)!} \mu^{(m,l)}(d \tilde{\lambda}_1,..., d \tilde{\lambda}_m)  \times \\
	& \qquad \qquad \qquad \qquad \qquad \qquad  \exp \left[  (N-m) \sum_{i=1}^m \Psi(\tilde{L}_{N-m},\tilde{\lambda}_i) \right]\Pa_{N-m,k-l}( d \tilde{\lambda}_{m+1},..., d \tilde{\lambda}_{N})  
	\end{align*}
	where 
	\begin{equation*}
	\tilde{L}_{N-m} = \frac{1}{N-m} \sum_{j=m+1}^{N} \delta_{\tilde{\lambda}_j(X)}.
	\end{equation*}
	The $1 / (k! (\frac{N-k}{2})!)$ factor arises from removing the ordering on the eigenvalues by real parts and the $(k-l)! (\frac{N-m+l-k}{2})!$ factor comes from ordering the last $N-m$ eigenvalues by real parts with the assumption that $k-l$ of them are real.  For $\epsilon > 0$, let $\mathbb{B}_{\epsilon} \subset \mathscr{P}(B_M)$ be the ball of radius $\epsilon$ around $U_{\tau}$ and $\mathbb{B}^{c}_{\epsilon}$ its complement. Since on the set $\{ \tilde{\lambda}_m \in [x,2M], \max_j \textnormal{Im }\tilde{\lambda}_j(X) \leq 2M \}$, we have 
	
	\begin{equation*}\exp \left[(N-m) \sum_{i=1}^m \Psi(\tilde{L}_{N-m}, \tilde{\lambda}_i) \right] \leq (4 \sqrt{2}M)^{m(N-m)},\end{equation*} we can further bound the exponential factor according by whether $\tilde{L}_{N-m}$ is in $\mathbb{B}_{\epsilon}$ or not, i.e, $$ \exp \left[ (N-m) \sum_{i=1}^m \Psi(\tilde{L}_{N-m}, \lambda_i) \right]\leq \exp \left[ m(N-m)\sup_{\mu \in \mathbb{B}_{\epsilon}, x \leq \textnormal{Re } z \leq M} \Psi(\mu,z) \right] + (4\sqrt{2}M)^{m(N-m)} \mathbbm{1}_{\mathbb{B}^{c}_{\epsilon}}(\tilde{L}_{N-m}).$$
	We can now perform the integral with respect to $\tilde{\lambda}_1,... \, ,\tilde{\lambda}_m$. Since we have removed all appearances of the variables of integration from the integrand and we are integrating over a compact region, the integral is finite for any $l$, and since $l$ belongs to a finite set, we can bound it by a constant $\gamma$ independent of $N$, $l$, and $k$. Since $\Pa_{N-m,k-l}(A) \leq \Pa_{N-m}(A)$ for any set $A$, we can also replace $\Pa_{N-m,k-l}$ by $\Pa_{N-m}$.
	
	We continue by summing over admissible $l$. Since we've eliminated any dependencies on $l$, we can replace the sum over $l$ by a factor of $k$.  By integrating over the remaining portion of our domain of integration and using the fact that $$\frac{(k-l)! \left( \frac{N-m+l-k}{2} \right)!}{k!\left(\frac{N-k}{2}\right)!} \leq 1,$$ we obtain the following upper bound: \begin{align*}
	&\Pa_{N,k} \left( \lambda_m \in [x,M], |\textnormal{Re } \lambda_1| \leq M, \max_j \Ima \lambda_j \leq M \right)\leq k\gamma \frac{K_{N-m}(\tau)}{K_N(\tau)} \left(\frac{N-m}{N}\right)^{\frac{N(N+1)}{4}} \times \\ &(4\sqrt{2}M)^{m(m-1)/2}\left( \exp \left[ m(N-m)\sup_{\mu \in \mathbb{B}_{\epsilon}, x \leq \textnormal{Re } z \leq M} \Psi(\mu,z) \right] + (4\sqrt{2}M)^{m(N-m)} \Pa_{N-m}(\tilde{L}_{N-m} \in \mathbb{B}^{c}_{\epsilon}) \right).
	\end{align*} The only dependence on $k$ on the right hand side comes from the factor $k$, so when we sum over $k$ we can just replace it with $N(N-1)/2$. At this we point, we make the following two observations: the first is that $\tilde{L}_{N-m}$ satisfies the same LDP as $L_{N}$. This implies that there exists $c > 0$ such that $$\Pa_{N-m}(\tilde{L}_{N-m} \notin \mathbb{B}_{\epsilon}) \leq e^{-cN^2}$$ hence the second term on the right hand side of our upper bound is negligible in the limit. The second observation is that from (\ref{K}) we have
	\begin{equation}\label{C-limalt}
	\lim_{N \rightarrow \infty} \frac{1}{N} \log \left[ \frac{K_{N-m}(\tau)}{K_N(\tau)} \left(\frac{N-m}{N}\right)^{\frac{N(N+1)}{4}} \right] = \frac{m}{2}.
	\end{equation}
	We use these two observations to establish the following inequality: \begin{equation*} 
	\limsup_{N \rightarrow \infty} \frac{1}{N} \log \Pa_{N} \left( \lambda_m \in [x,M], |\textnormal{Re } \lambda_1| \leq M, \max_j \Ima \lambda_j \leq M \right) \leq \frac{m}{2} + m\lim_{\epsilon \downarrow 0} \sup_{\mu \in \mathbb{B}_{\epsilon}, x \leq \textnormal{Re } z < M} \Psi(\mu,z).
	\end{equation*}
	The second term on the right hand side can by computed explicitly, $$\lim_{\epsilon \downarrow 0} \sup_{\mu \in \mathbb{B}_{\epsilon}, x \leq \textnormal{Re } z \leq M} \Psi(\mu,z) = \Psi(U_{\tau},x) = -I_{\tau}(x) - \frac{1}{2}$$ where the first equality follows from upper semi-continuity of $\Psi$ and {\sc Corollary}  \ref{logpotineq} and the second equality follows from (\ref{ratefunc}). This proves the upper bound for the equality (2) stated at the beginning of the proof.
	
	To obtain the lower bound, we fix $y > x > r > 1 + \tau$ and $\epsilon, \delta > 0$. We first need a lower bound analogous to (\ref{erfc}). We can obtain one if we restrict ourselves to $|\Ima z| \leq \delta$ and $N$ large enough: \begin{equation*} \exp \left[ -\frac{N-m}{2(1+\tau)} (z^2 + \bar{z}^2) \right] \textnormal{erfc} \left[ \sqrt{\frac{2(N-m)}{1-\tau^2}} |\Ima z| \right]  \geq \frac{\beta(\delta)}{\sqrt{N}} \exp \left[ \frac{-(N-m)|z|^2}{1 - \tau} \right]\end{equation*}for some positive constant $\beta(\delta) < 1$. Retaining the previous notation from the upper bound, we further define $\mathbb{B}_{\epsilon} \cap \mathscr{P}(B_r)$ to mean the set of measures in $\mathbb{B}_{\epsilon}$ whose support is contained in the ball $B_r \subset \C$ of radius $r$.  On the set  
	
	\begin{align*}
	\left \{ \tilde{\lambda}_m(X) \in \left[ \sqrt{\frac{N}{N-m}}x,y \right], \textnormal{Im }\tilde{\lambda}_j(X)\leq  \delta \, , |\textnormal{Re } \lambda_1(X)| \leq y, |\lambda_j(X)| \leq r \, \forall j \right \} 
	\end{align*}
	which is a subset of $\left \{ \lambda_m(X) \in [x,M], \max_j \Ima \lambda_j(X),|\textnormal{Re } \lambda_1(X)| \leq M \right \}$,
	we can bound the density as follows:

	\begin{align*}
	\Pa_{N,k}(d \lambda) &= \frac{K_{N-m}(\tau)}{K_N(\tau)} \sum_{l} \prod_{1 \leq i < j \leq m} | \tilde{\lambda}_j - \tilde{\lambda}_i|  \prod_{j=1}^m e^{-\frac{N-m}{2(1+\tau)} \tilde{\lambda}^2_j} \sqrt{\textnormal{erfc}\left(\sqrt{\frac{2(N-m)}{1-\tau^2}} |\Ima \tilde{ \lambda}_j|\right)} \, \mu^{(m,l)}(d \tilde{\lambda}_1,...,d\tilde{\lambda}_m) \times \\
	& \qquad \qquad \qquad \frac{(k-l)! \left( \frac{N-m+l-k}{2} \right)!}{k!\left(\frac{N-k}{2}\right)!}\left(\frac{N-m}{N}\right)^{\frac{N(N+1)}{4}}  \prod_{i=1}^m \prod_{j=m+1}^{N} |\tilde{\lambda}_i - \tilde{\lambda}_j|  \Pa_{N-m,k-l}(d\tilde{\lambda}_{m+1},...,d\tilde{\lambda}_N) \\ 
	&\geq \sum_l \left(\frac{\beta(\delta)}{\sqrt{N}} \right)^{\frac{m-l}{2}} \left(\frac{N-m}{N}\right)^{\frac{N(N+1)}{4}} \frac{(k-l)! \left( \frac{N-m+l-k}{2} \right)!}{k!\left(\frac{N-k}{2}\right)!} \prod_{1 \leq i < j \leq m} |\tilde{\lambda}_j - \tilde{\lambda}_i| \,  \mu^{(m,l)}(d \tilde{\lambda}_1,..., d \tilde{\lambda}_m)  \times \\
	&  \frac{K_{N-m}(\tau)}{K_N(\tau)} \mathbbm{1}_{\mathbb{B}_{\epsilon} \cap \mathcal{P}(B_r)}(\tilde{L}_{N-m}) \exp \left[  m(N-m) \inf_{ \substack{ \mu \in \mathbb{B}_{\epsilon} \cap \mathscr{P}(B_r) \\ x \leq \textnormal{Re }z \leq y,|\Ima z| < \delta}}  \Psi(\mu,z) \right]\Pa_{N-m,k-l}(d \tilde{\lambda}_{m+1},...,d \tilde{\lambda}_N) \\  
	\end{align*}
	We now point out two quantities which will end up becoming negligible in the limit. We first proceed by integrating out the $\tilde{\lambda}_1,... \, , \tilde{\lambda}_m$ variables which yields a finite quantity since we are integrating over a bounded region and the integrand is bounded. Moreover, this quantity is bounded both from above and from below by constants independent of $N$ so this term will be neglible in the limit.  The second quantity which is also negligible in the limit is the one appearing in the following limit which holds for $l \leq m$ and all $k$: \begin{equation}\label{C-lim}
	\lim_{N \rightarrow \infty} \frac{1}{N} \log \left[ \frac{(k-l)! \left( \frac{N-m+l-k}{2} \right)!}{k!\left(\frac{N-k}{2}\right)!} \right] = 0.
	\end{equation}  Since the inequality $$\sum_k \sum_l \Pa_{N-m,k-l}(A) \geq \Pa_{N-m}(A)$$ is true for any Borel set $A$, we can use (\ref{C-limalt}) and (\ref{C-lim}) to obtain the following lower bound:
	\begin{align*} \liminf_{N \rightarrow \infty}\frac{1}{N} \log \Pa_N \left(\lambda_m \in [x,M] \right) \geq \frac{m}{2} + m \lim_{\epsilon \downarrow 0} \inf_{ \substack{ \mu \in \mathbb{B}_{\epsilon} \cap \mathscr{P}(B_r) \\ x \leq \textnormal{Re }z \leq y,|\Ima z| < \delta}}  \Psi(\mu,z)
	\end{align*}
	By continuity of $\Psi$ and {\sc Corollary} \ref{logpotineq}, we obtain: $$\lim_{\epsilon \downarrow 0} \inf_{ \substack{ \mu \in \mathbb{B}_{\epsilon} \cap \mathscr{P}(B_r) \\ x \leq \textnormal{Re }z \leq y,|\Ima z| < \delta}}  \Psi(\mu,z) = \Psi(U_{\tau}, y + \delta i).$$ Finally, we take $\delta \rightarrow 0$ then $y \rightarrow x$  and use the continuity of $\Psi$ coupled with (\ref{ratefunc}) to obtain the desired lower bound for equality (2) stated at the beginning of the proof.
	
\end{proof}

\section{Expected number of critical points and the Gaussian Elliptic Ensemble}

 We now relate $\mathcal{N}_{m}$ to the eigenvalue with $m$th largest real part of an $N \times N$ GEE matrix. \begin{theorem}\label{eq:aux}For a Borel set $B \subset \R$, we have:
	\begin{equation*}
	\E\mathcal{N}_{m}(B) = 2\sqrt{\frac{1+\tau}{b^2+\tau}} b^{1-N} \E_N \left[ \exp \left( -\frac{N(1-b^2)}{2(b^2 + \tau)(1+\tau)} \lambda_m^2(X) \right) \mathbbm{1}_{B} \left(\sqrt{\Phi'_1(1)}\lambda_m(X) \right) \right].
	\end{equation*}
\end{theorem}
The proof of {\sc Theorem} \ref{eq:aux} will follow from two results. The first relates $\E \mathcal{N}_m(B)$ to a matrix integral through the Kac-Rice formula adapted to our setting. \begin{theorem}\label{eq:KR} For a matrix $A$ and nonnegative integer $m$, set $$i_m(A) = \begin{cases}  1 & \tif A \textnormal{ has exactly } m \textnormal{ eigenvalues with nonnegative real part } \\
0 & \tif \textnormal{else}\end{cases}$$then
$$
	\E \mathcal{N}_{m}(B) = \frac{2\sqrt{N-1}^{N}}{2^{N/2} \Gamma(N/2)} \frac{b^{1-N}}{\sqrt{b^2 + \tau}}\int_{-\infty}^{\infty} \mathbbm{1}_{\sqrt{\frac{N}{N-1}}B} \left( \sqrt{\Phi'_1(1)} \lambda \right) e^{-\frac{(N-1)\lambda^2}{2(b^2+\tau)}}  \E_{N-1} \left[ \left|\det \left ( X - \lambda  I \right) \right| i_m(X - \lambda I) \right] d \lambda.$$
\end{theorem}
 The second result relates the complicated integral against $\Pa_{N-1}$ appearing in {\sc Theorem } \ref{eq:KR} to a simpler one against $\Pa_N$.
 
 \begin{lemma}\label{uppingDim}
 	For any bounded Borel measurable function $f$ on $\R$, we have \begin{align*}
 	\int_{-\infty}^{\infty}f (t \sqrt{N-1}) \exp \left(- \frac{N-1}{2(1+\tau)} t^2 \right)& \E_{N-1} \left[ \left| \det \left(X - t I \right) \right| i_m(X - t I) \right] d t	\\  &= \frac{\Gamma(N/2)\sqrt{2}^N  \sqrt{1+\tau} }{\sqrt{N-1}^{N}}\E_{N}[\mathbbm{1}_{ \R}(\lambda_{m+1}(X)) f(\sqrt{N} \cdot \lambda_{m+1}(X))].
 	\end{align*}
 \end{lemma}
 
 Given these two results, the proof of {\sc Theorem} \ref{eq:aux} goes as follows. 
 
 \begin{proof}[Proof of Theorem \ref{eq:aux}]If we apply {\sc Lemma} \ref{uppingDim} to the function $$f(x) = \mathbbm{1}_{\sqrt{N}B}\left(\sqrt{\Phi'_1(1)}x \right)\exp \left(- \frac{1-b^2}{2(1+\tau)(b^2 + \tau)} x^2 \right)$$ then we can use the resulting equality to simplify the formula in {\sc Theorem} \ref{eq:KR} to recover the expression on the right of {\sc Theorem} \ref{eq:aux} and hence conclude the result. 
 
 \end{proof}

We relegate the proof {\sc Theorem} \ref{eq:KR} to the next section and finish the current subsection with a proof of {\sc Lemma} \ref{uppingDim}. 
 \begin{proof}[Proof of Lemma \ref{uppingDim}]
 	
 	We first remark that $\Delta(\lambda(X),t) = |\det(X - t I)| \Delta(\lambda(X))$. Next, note that the factor $i_m(X - t I) = 1$ if and only if we have the following inequality:
 	\begin{equation*}
 	\textnormal{Re } \lambda_1 > ... >  \textnormal{Re } \lambda_{m} >  t > ... > \textnormal{Re } \lambda_{N-1};
 	\end{equation*}
 	otherwise it is 0. These two remarks suggest that $t$ can fit in nicely as a (real) eigenvalue of a larger GEE matrix. If we restrict to the case of only $k$ real eigenvalues and if we relabel $t$ as $\lambda_{m}$ and $\lambda_{j} := \lambda_{j+1}$ for $j \geq m$ then we can rewrite $d \mu^{(N-1,k)} d t = d \mu^{(N,k+1)}$ and expand the left hand side of {\sc Lemma} \ref{uppingDim} as
 	
 	\begin{equation*}
 	\frac{K_N(\tau)}{K_{N-1}(\tau)}\sum_k  \int f(\lambda_m \sqrt{N-1}) \frac{| \Delta(\lambda)|}{K_{N}(\tau)}  \exp \left(-\frac{N-1}{1+\tau}\sum_{j=1}^N \frac{\lambda^2_j}{2} \right) \prod_{j=1}^N \sqrt{\textnormal{erfc} \left( \sqrt{\frac{2(N-1)}{1-\tau^2}} |\Ima \lambda_j| \right)} \mu^{(N,k+1)}(d \lambda)
 	\end{equation*}
 	where the integral is over the appropiate domain. The factor to the right of $f$ looks exactly like the density for $\Pa_{N,k+1}$ except with an implicit factor of $\mathbbm{1}_{\R}(\lambda_m)$ since we are mandating that $\lambda_m$ be real and the fact that we have $N-1$ instead of $N$ scattered in the density. We can remedy the latter issue by performing a substitution $\lambda := \sqrt{ \frac{N}{N-1}} \lambda$. Following the substitution, we obtain the left hand  side of {\sc Lemma} \ref{uppingDim} is equivalent to the following expression:
 	
 	\begin{equation*}
 	\frac{K_N(\tau)}{K_{N-1}(\tau)} \sqrt{ \frac{N}{N-1}}^{N + \binom{N}{2}}\E_{N}[\mathbbm{1}_{ \R}(\lambda_{m+1}(X)) f(\sqrt{N} \cdot \lambda_{m+1}(X))].
 	\end{equation*}
 	A simple algebra computation using (\ref{K}) reveals the leading constant is exactly as stated in the lemma.
 	
 \end{proof}
  \section{Proof of Theorem \ref{eq:KR} }
  
  The proof of this theorem will be broken up into a series of steps.
  
   Our first step is to invoke the traditional Kac-Rice formula. In order to do so, we will establish some notation. Given an equilibrium point, we choose coordinates in a neighborhood around $\sigma$ so that we can write $\sigma = 0$ and $F(0)$ as a random vector in $\R^{N-1}$.  We define $\rho_{F(\sigma)}$ to be the density function for the random vector $F(0)$. This depends on the choice of coordinates, but its value at 0 does not. Through the use of local coordinates, the classical Kac-Rice formula (see e.g. {\sc Theorem} 6.2 in Aza{\"i}s and Wschebor ~\cite{AW}) yields the following formula for $\E \mathcal{N}_m(B)$: 
\begin{equation} \label{KR}\E \mathcal{N}_{m}(B) = \int_{S^{N-1}(\sqrt{N})} \E[ |\det JF(\sigma)| i_m(JF(\sigma)) \mathbbm{1}_{B}(\lambda(\sigma)) | F(\sigma) = 0 ] \rho_{F(\sigma)}(0) d\sigma.
\end{equation}

The second step is to exploit the large symmetry group of the sphere, the orthonormal group $O(N)$, and its relationship with the integrand. It will allow us to reduce the integral in (\ref{KR}) to the integrand evaluated at the point $$\mathbf{n} = (0,...,0,\sqrt{N}) \in S^{N-1}(\sqrt{N}) \subset \R^{N}$$ times a factor of $\vol (S^{N-1}(\sqrt{N}))$. 

\begin{lemma} The function
$$\sigma \mapsto \E \left[|\det JF(\sigma)| \mathbbm{1}_{B}(\lambda(\sigma)) i_m(JF(\sigma)) | F(\sigma) = 0 \right] \rho_{F(\sigma)}(0)$$is invariant under the standard $O(N)$ action on $S^{N-1}(\sqrt{N})$ and hence is constant.
\end{lemma}

\begin{proof}
The key point is that the only probabilistic portion of $JF$ comes from $f$, $h$, and the partial derivatives of $f$ with respect to the ambient $\R^N$ variables. To see this, we define $j : S^{N-1}(\sqrt{N}) \rightarrow \R^N$ to be the usual embedding and for $x \in S^{N-1}(\sqrt{N}) \subset \R^N$. Define $\proj_{x} : \R^N \rightarrow T_x S^{N-1}(\sqrt{N})$ to be the standard projection. Then

\begin{equation*}
JF(x) =  \proj_x \circ J_{euc}F_x \circ dj_x
\end{equation*}
where $J_{euc}F|_x := J_{euc}F(x) := \left( \frac{\partial F_j}{dx_i}(x) \right)$ is the Jacobian of $F$ (viewed as a function from $\R^N$ to $\R^N$) at $x$ and $dj_x$ is the differential of $j$ at $x$. There is a simple relationship between the $O(N)$ action and the functions $\proj_x$ and $dj_x$. For any $g \in O(N)$, $g \, \proj_x \, g^{-1} = \proj_{gx}$ and $dj_{gx} = g \, dj_{x} \, g^{-1}$. We can use this to write $JF(gx)$ as follows: 
\begin{align*}
JF(gx) &=  \proj_{gx} \cdot J_{euc} F |_{gx} \cdot dj_{gx} \\
&= g \, \proj_{x} \left( g^{-1} \cdot J_{euc} F |_{gx} \cdot g \right) dj_{x} \, g^{-1}.
\end{align*}
We aim to prove that $(F(x),J_{euc}F |_x) = (g^TF(gx), g^{-1} \cdot J_{euc}F |_{gx} \cdot g)$ in distribution. The lemma will follow from this claim since conditioning on $F(gx) = 0$ is equivalent to $g^T F(gx) = 0$ and by orthogonality of $g$, $g^T = g^{-1}$. To obtain the required equality, we first write out $J_{euc}F|_x$ and $g^{-1} \cdot J_{euc}F|_{gx} \cdot g$ in terms of the ambient $\R^N$ coordinates:

\begin{align*}
\left( J_{euc}F|_x \right)_{ij} &= - \frac{\partial \lambda}{\partial x_j}(x) x_i - \lambda(x) \delta_{ij}+  \frac{\partial f_i}{\partial x_j}(x) \\
\left( g^{-1} \cdot J_{euc} F |_{gx} \cdot g \right)_{ij} &= -  (g^T \nabla \lambda(gx))_j x_i  - \lambda(gx) \delta_{ij} - \left( g^T \cdot J_{euc}f |_{gx} \cdot g \right)_{ij},
\end{align*}
where $\nabla \lambda$ is the gradient of $\lambda$. We can rewrite the expressions involving $\lambda$ in terms of $h$, $f$ and its derivatives as follows:

\begin{align*}
\lambda(gx) &= \frac{1}{N} \langle gx , f(gx) + h \rangle = \frac{1}{N} \langle x, g^T f(gx) + g^Th \rangle \\
\frac{\partial \lambda}{\partial x_j}(x) &= \frac{1}{N} \left(f_j(x) + h_j \right) + \langle x, j^{th} \textnormal{ column of } J_{euc}f |_{x}  \rangle \\
g^TF(gx) &= \lambda(gx)x + \frac{1}{N} \left( g^Tf(gx) + g^Th \right)\\
\left(g^T \nabla \lambda(gx) \right)_j &= \frac{1}{N} \left( g^Tf_{j}(gx) + g^Th_j \right) + \langle x, j^{th} \textnormal{ column of } g^T \cdot J_{euc} f |_{gx} \cdot g \rangle 
\end{align*}
From Equation (3.17) in Fyodorov~\cite{Fyo2}, we know that $\left(f(x), J_{euc}f |_x \right) = \left( g^Tf(gx), g^{T} \cdot J_{euc}f |_{gx} \cdot g \right)$ in distribution. Since $h$ is Gaussian, then $h = g^Th$ in distribution and thus by independence, we have $$\left(f(x), J_{euc}f |_x,h \right) = \left( g^Tf(gx), g^{T} \cdot J_{euc}f |_{gx} \cdot g , g^Th \right)$$ in distribution which implies what was desired. 
\end{proof}

The third step is to employ explicit coordinates around $\mathbf{n}$ to write down a formula for $JF$ and $F$. Let $B_{\sqrt{N}}$ denote the ball centered around $0 \in \R^{N-1}$ of radius $\sqrt{N}$ and define the map $P_N : B_{\sqrt{N}} \rightarrow S^{N-1}(\sqrt{N})$ by $$P_N(x_1,...,x_{N-1}) =  \left(x_1,...,x_{N-1}, \sqrt{N - |x|^2} \right)$$where $|x|^2 := \sum_{i=1}^{N-1} x^2_i$. With these coordinates, it is easy to compute formulas for $dj_{\n}$ and $\proj_{\n}$, yielding:
 
 \begin{equation*}
JF(\n) = \left( J_{euc}F_{ij} \right)_{i=1,j=1}^{N-1,N-1} = \left( \frac{\partial f_j}{\partial x_i} (\n) - \lambda(\n) \delta_{ij} \right)_{i=1,j=1}^{N-1,N-1}
 \end{equation*}
Note that the indices go up to $N-1$ and not up to $N$. Next, we compute $F(0)$ in coordinates as a vector in $\R^{N-1}$. We will compute it by establishing a choice of basis vectors for $T_{\n}S^{N-1}(\sqrt{N})$. Our basis vectors $\{ v_i \}$ will be the pushforward of the basis in $\R^{N-1}$ through our map $P_N$ i.e.,  $v_i = dP_N(e_i)$ where $\{ e_i \}$ is the standard basis vectors for $\R^{N-1}$. In this basis, we can write $F(0)$ as $$F(0) = \left( f_i(\mathbf{n}) + h_i \right)_{i=1}^{N-1}$$ and $\lambda(\n)$ as $$\lambda( \n) = \frac{f_N(\n) + h_N}{N}$$We also remark that conditioning on $F(\mathbf{n}) = 0 \in \R^N$ is the same as conditioning on $F(0) = 0 \in \R^{N-1}$. 

Our fourth step is to relate our random matrix integral to the Gaussian Elliptic Ensemble. To that end, we now make four assertions which we leave to the reader to verify:
 \begin{enumerate}
 	\item $\frac{\partial f_i}{\partial x_j}(\mathbf{n})$ is independent of $f_N(\mathbf{n})$ for $i,j \leq N-1$.
 	\item $F(0)$ is independent of $J F(\mathbf{n})$.
 	\item $f_N(\mathbf{n}) + h_N$ is a mean zero Gaussian with variance $\Phi_1(1) + \Phi_2(1) + \sigma^2$. 
 	\item For $1 \leq i ,j,n,m \leq N-1$, we have:$$\frac{N}{(N-1)\Phi'_1(1)}\E[\partial_j f_i(\mathbf{n}) \partial_{n} f_m(\mathbf{n})] = \frac{1}{N-1} \left( \delta_{in}\delta_{jm} + \frac{\Phi_2(1)}{\Phi'_1(1)} \delta_{im} \delta_{jn}\right).$$
 	\end{enumerate}
 	 Through the use of assertion (4) and the formula for $JF(\mathbf{n})$, we can write $JF(\mathbf{n})$ in terms of the Gaussian Elliptic Ensemble: $$\sqrt{\frac{N}{(N-1)\Phi'_1(1)}} JF(\mathbf{n}) = X - ZI$$
 	 in distribution, where $X$ has the law $\Pa_{N-1}$ and $Z$ is a Gaussian random variable independent of $X$ with mean 0 and variance given by $$\frac{\sigma^2  + \Phi_1(1) + \Phi_2(1)}{(N-1)\Phi'_1(1)} = \frac{b^2 + \tau}{N-1}.$$By the independence of $h$ from $f$, we know that $F(0)$ consists of $N-1$ independent mean zero Gaussian random variables with variance $\Phi_1(1) + \sigma^2$ and hence $$\rho_{F(\mathbf{n})}(0) =(2 \pi (\Phi_1(1) + \sigma^2))^{-(N-1)/2}.$$Summarizing all the steps we have taken, we can conclude the following expression for $\E \mathcal{N}_{m}(B)$: $$\E \mathcal{N}_{m}(B) = \E \left[| \det(X - ZI)|  \mathbbm{1}_{\sqrt{\frac{N}{N-1}}B} \left( \sqrt{\Phi'_1(1)} Z \right) \right] \frac{\vol(S^{N-1}(\sqrt{N}))}{\sqrt{2 \pi b^2}^{N-1}} \left( \frac{N-1}{N} \right)^{\frac{N-1}{2}}. $$Finally, using the formula for the volume of a sphere, $$\vol \left(S^{N-1}(\sqrt{N}) \right) = \frac{2 \pi^{N/2}}{\Gamma(N/2)}\sqrt{N}^{N-1}$$ and the explicit density function of $Z$, we obtain the expression on the right hand side of {\sc Theorem} \ref{eq:KR}. This completes the proof. 
\section{Proof of the main results}

 In this section, we prove the main results stated in {\sc Section} 1. {\sc Theorem} \ref{largedev} and {\sc Theorem} \ref{eq:aux}, coupled with Varadhan's lemma (see {\sc Theorem} 4.3.1 of Dembo and Zeitouni~\cite{DZ}) yields {\sc Theorem} \ref{lagmult}. Setting $c = -\infty$ and $d = \infty$ in {\sc Theorem} \ref{lagmult} results in {\sc Theorem} \ref{mainthm}. Finally, {\sc Theorem} \ref{divind} is a trivial corollary of the following lemma:

\begin{lemma}\label{divindex1}
Define $m(N)$ to be a sequence integers such that $\frac{m(N)}{N} \rightarrow \gamma \in (0,1)$ and let $\epsilon > 0$. Then, there exists a constant $c := c(\epsilon) > 0$ such that $$\Pa_N \left( \textnormal{Re } \lambda_{m(N)} \notin (s_{\gamma} - \epsilon, s_{\gamma} + \epsilon) \right) \leq \exp(-cN^2).$$
\end{lemma}

\begin{proof}
This is an immediate consequence of the fact that $L_N$ satisfies a large deviation principle with speed $N^2$ whose rate function is minimized at  $U_{\tau}$. The proof of this LDP for the case $\tau = 0$ can be found in Ben Arous and Zeitouni~\cite{BGZ} but the same argument extends to $\tau \in (-1,1)$.	
	
We first break up the left hand side of the equality as follows: $$\Pa_N \left( \textnormal{Re }\lambda_{m(N)}(X) \notin (s_{\gamma} - \epsilon, s_{\gamma} + \epsilon) \right) = \Pa_N \left( \textnormal{Re } \lambda_{m(N)} < s_{\gamma} - \epsilon \right) + \Pa_N \left( \textnormal{Re }\lambda_{m(N)} > s_{\gamma} + \epsilon \right).$$
To estimate the first term, we let $L_N := L_N(X_N) := \frac{1}{N} \sum_{i=1}^{N} \delta_{\lambda_i(X_N)}$ be the empirical distribution of the eigenvalues of a matrix $X_N$ with law $\Pa_N$. Given the aforementioned LDP, we have:$$\Pa_N( \textnormal{Re }\lambda_{m(N)} > s_{\gamma} + \epsilon) = \Pa_N \left(L_N( z : \textnormal{Re } z > s_{\gamma} + \epsilon) \geq \frac{m(N)}{N} \right) \leq \frac{1}{2}\exp(-cN^2)$$ for some $c > 0$ since $U_{\tau}( z : \textnormal{Re } z > s_{\gamma} + \epsilon) < \gamma$. Similarly,$$\Pa_N(\textnormal{Re } \lambda_{m(N)} < s_{\gamma} - \epsilon) = \Pa_N \left(L_N(z : \textnormal{Re } z > s_{\gamma} - \epsilon) \leq \frac{m(N)-1}{N} \right) \leq \frac{1}{2}\exp(-cN^2).$$
\end{proof}

\end{document}